\documentclass[a4paper,12pt,twoside]{article}
\usepackage[latin1]{inputenc}
\usepackage[dvips]{graphics}
\input xy
\input cyracc.def

\xyoption{all}
\usepackage{xr}
\usepackage{amssymb}
\usepackage{amsmath}
\usepackage{amsthm}
\usepackage{amsfonts}
\usepackage{mathrsfs}
\theoremstyle{plain}
\newtheorem{T}{Theorem}[section]

\newtheorem{Cor}[T]{Corollary}
\newtheorem{TL}[T]{Lemma}
\newtheorem{Prop}[T]{Proposition}

\newcommand\mps[1]{\marginpar{\small\sf }}

\theoremstyle{definition}
\theoremstyle{remark}
\theoremstyle{remark}\newtheorem{remark}{Remark}[T]

\newcommand{\bb}[1]{\mathbb{#1}}
\newcommand\calo{\mathcal O}
\newcommand\CH{\op{CH}}

\newcommand{\mc}[1]{\mathcal{#1}}

\newcommand{\lci}{local complete intersection }

\newcommand\op[1]{\operatorname{#1}}

\newcommand\rk{\operatorname{rk}}

\newcommand\spec{\operatorname{Spec}}

\input cyracc.def

\theoremstyle{plain}
\newtheorem{example}{Example}
\theoremstyle{definition}
\newcommand{\Spec}{\mathrm{Spec\,\,}}
\newcommand{\ses}[5]{$$
\xymatrix@1{ 0 \ar[r] & {{#1}} \ar[r]^-{{#2}} & {{#3}} \ar[r]^-{{#4}} &
{{#5}} \ar[r] & 0
\\ }$$}

\newcommand{\sesdot}[5]{$$
\xymatrix@1{ 0 \ar[r] & {{#1}} \ar[r]^-{{#2}} & {{#3}} \ar[r]^-{{#4}} &
{{#5}} \ar[r] & 0. \\ }$$}

\newcommand{\sesbig}[5]{$$\xymatrix@1{{\raisebox{1.0ex}[3.0ex][1.0ex]{$0$}}
\ar@<0.6ex>[r] & {\raisebox{1.0ex}[3.0ex][1.0ex]{${#1}$}}
\ar@<0.6ex>[r]^{#2} & {\raisebox{1.0ex}[3.0ex][1.0ex]{${#3}$}}
\ar@<0.6ex>[r]^{#4} & {\raisebox{1.0ex}[3.0ex][1.0ex]{${#5}$}}
\ar@<0.6ex>[r] & {\raisebox{1.0ex}[3.0ex][1.0ex]{$0$}}}$$}

\newcommand{\Comment}[1]{}

\title{Discriminants and Artin conductors}
\author{Dennis Eriksson}
\begin{document}

\maketitle
\abstract{We study questions of multiplicities of discriminants for degenerations coming from projective duality over discrete valuation rings. The main result is a type of discriminant-different formula in the sense of classical algebraic number theory, and we relate it to Artin conductors via Bloch's conjecture. In the case of discriminants of planar curves we can calculate the different  precisely. In general these multiplicities encode topological invariants of the singular fibers and in the case of characteristic $p$, wild ramification data in the form of Swan conductors.} 
\tableofcontents 

\newpage
\section{Introduction}

The theory of discriminants is an old field which was recently re-animated by the beautiful work of Gelfand-Zelevinsky-Kapranov (cf. \cite{GZK}). Discriminants are classically polynomials which vanish along singular varieties. A recurrent example is that of a quadratic binary or ternary form. The discriminant of such a polynomial is then a degree 2 (resp.  3) hypersurface in $\bb P^2$ (resp. $\bb P^5$). More precisely, if
$$F = aX^2 + bXY + cY^2$$
then $\Delta_F = b^2 - 4ac$ and
if $$F= a_1X^2 + a_2Y^2 + a_3Z^2 + a_4 XY + a_5XZ + a_6 YZ$$ then
$$\Delta_F = a_1 a_6^2 + a_5^2 a_2 + a_4^2 a_3-a_4a_5a_6 - 4 a_1a_2a_3.$$
%In general, for a diagonalized quadratic $$\sum a_i X_i^2$$ the discriminant is simply $-4\prod a_i$. \\
A motivating example of this article is the following. Let $E$ be an elliptic curve over a discretely valued field $K$ with ring of integers $R$ which is henselian with algebraically closed residue field. It is given, by a Weierstrass equation $$y^2 + a_1xy + a_3 y = x^3+ a_2 x^2 + a_4 x + a_6.$$
The discriminant of the equation is, for $a_1 = a_3 = a_2 = 0$, given by $\Delta =-16(4{a_4}^3 + 27{a_6}^2)$. Given a model $\mc W$, over $R$, i.e. $a_i \in R$, there is a special model whose discriminant is minimal with respect to the discrete valuation, the minimal Weierstrass equation. A famous formula of Ogg (\cite{Ogg}, also see \cite{T.Saito-conductor} for a more general result and whose proof also repairs a gap in mixed characteristic $(0,2)$ in the original article) states that for this minimal Weierstrass model we have
$$\op{ord} \Delta = \deg c_{2,\mc E}^{\mc E_s}(\Omega_{\mc E/S}) = -\op{Art}_{\mc E/S} = m_E - 1 + f_E $$
where $m_E$ is the number of irreducible components in the N\'eron model $\mc E$ (resp. $f_E$ the exponent of the conductor) of $E$ over $R$. See below for the other two terms. From the point of view of computing the conductor, this formula is very powerful since Tate's algorithm and can be implemented on a computer actually allows us to find the minimal Weierstrass model. Another famous formula is the F\"uhrer-Diskriminanten-produkt formula in classical number theory which relates the discriminant of a finite extension of local fields to the conductor. They are related through the different by:
$$\op{ord} \Delta_{L/K} = \op{ord} \op{Norm}_{L/K}(\delta_{L/K}) = \op{Art}_{L/K}$$
Both of these formulas relate different ways of measuring singularities, the discriminant is somewhat of a geometric object, whereas the conductor is an object built out of monodromy. The different in turn is an object constructed out of the K\"ahler differentials.  In \cite{Tate}, p. 192, Tate asks about Ogg's formula: "It would be interesting to know what is behind this mysterious equality". It seems that the first one to consider this connection in higher dimension was Deligne in \cite{SGA7-2}, Expos\'e XVI, where he conjectures an equality between the Milnor number of a point and the (total) dimension of the vanishing cycles in the same point. In \cite{blochcycles}, Bloch then conjectured a relation between a localized Chern class (see next section for the definition) and the Artin conductor, for regular schemes over a discrete valuation ring, which would correspond to the total dimension of the vanishing cycles and the different in classical number theory. This states that if $X$ is regular and $X \to S$ is a flat generically smooth projective morphism of relative dimension $d$ with geometrically connected fibres, $S$ is the spectrum of a discrete valuation ring with generic (resp. special) point $\eta$ (resp. $s$, with perfect residue field $k(s)$ of characteristic $p \geq 0$), then
 $$\deg c_{d+1, X}^{X_s}(\Omega_{X/S}) = (-1)^{d} \op{Art}_{X/S}$$
 where
 $$\op{Art}_{X/S} := \op{Art}_{X'/S'} := \chi(X_{\overline{\eta}}) - \chi(X_{\overline{s}}) + \sum_{q=0}^{2d} (-1)^q \op{Sw} H^q(X_{\overline{\eta}}, {\bb Q}_{\ell})$$
 where $S'$ is the spectrum of the strict henselisation of $R$, $X' = X \times_S S'$, $\overline{s}$ and $\overline{\eta}$ are used to denote algebraic closures of the fields, $\chi$ denotes $\ell$-adic Euler-characteristic for $\ell \neq p$ and $\op{Sw}$ denotes the Swan conductor of the natural Galois representation acting on the various cohomology groups (cf. introduction of \cite{katosaito}, \cite{Serre-facteurs}, 2.1 or \cite{Corps-locaux}, chapter VI, for a general discussion on conductors). The major breakthrough in this field was the article \cite{katosaito} which proved this relation in full generality (assuming resolution of singularities). \\
An immediate generalization of Ogg's formula in the above sense is not obvious. The classical proof rests a lot on the geometry of Weierstrass equations as well as N\'eron's classification of the special fibers of models of elliptic curves. But a Weierstrass equation is not a general degree 3 ternary form. The discriminant of such objects form hypersurface of degree 12 in $\bb P^9$ defined by a polynomial with 2040 non-zero coefficients, and even if one restricts oneself to Weierstrass equations, this is a hypersurface in $\bb P^6$ with 26 coefficients. It is seems improbable that a generalization of multiplicities of discriminants in general can be obtained by considering specific equations, except in favorable situations. On the other hand, T. Saito's proof only uses the geometry of the Weierstrass equations, but is restricted to relative dimension one.\\
The aim of the current article is to find relations between the "different", i.e. the localized Chern class mentioned above, and the order of vanishing of the discriminant in the sense of projective duality. We recall the setting. Let $k$ be a field and $X$ a smooth geometrically integral variety of dimension $n+1$ over $k$ with a fixed $k$-embedding $X \subseteq \bb P^M$. Suppose furthermore that the image of $X$ is non-degenerate. Then we define the discriminant variety (or dual variety) as the subvariety $\Delta_X \subseteq \check {\bb P}^M$ defined by all the hyperplanes $H \in \check {\bb P}^M$ such that $X \cap H$ is singular. The variety $\{(x,H) \in X \times \check {\bb P}^M, x \in (X \cap H)_{\hbox{sing}}\}$ of singular hyperplanes is the projective bundle $\bb P(N)$ over $X$ where $N$ is the normal bundle of $X$ in $\bb P^M$. The map $\bb P(N)\to \check {\bb P}^M$ sending $(x,H)$ to $H$ has schematic image $\Delta_X$ and the map $\varphi: \bb P(N) \to \Delta_X$ is called the Gauss morphism. Over $\check{\bb P}^M$ we have the tautological hypersurface $\mc H$. Their relations are summed up in the commutative diagram \\
$$\xymatrix{ & \bb P(N) \ar[r] \ar[d]^{\varphi} \ar[ld] & \mc H \ar[d]^f \ar[r] & X \times \check {\bb P}^M. \ar[ld]^p \\
X & \Delta_X \ar[r] & \check {\bb P}^M }$$ In this article we will be concerned with the case when the variety $\Delta_X$ is a hypersurface so that the Gauss morphism is proper and generically finite. Then $\Delta_X$ is defined by a homogenous polynomial defined up to an element in $k^*$. When $X = \bb P^n$ and the projective embedding is the $d$-th Veronese embedding we can also make the same construction over the integers. \\
The main results are the following. Over a field $k$, we calculate localized Chern class of the tautological family of hypersurfaces over $\check{\bb P}^M$. As an application, we prove the following multiplicity formula for a degenerating family of hypersurface sections: 
 \begin{T} \label{thm:discriminant-different} [Discriminant-Different formula] Let $X$ be as above and suppose $\Delta_X$ is a hypersurface. Also suppose we are given a discrete valuation ring $R$ and a morphism $\spec R \to \check {\bb P}^M$ such that the generic point is not in the discriminant. Denote by $H$ the pullback of the tautological hyperplane section over $\check {\bb P}^M$, by $H_s$ the special fiber and by $\Delta_{X,R}$ the pullback of $\Delta_X$ to $S = \Spec R$. Then, for the discrete valuation $v$ on $R$,  $$\op{ord}_v \Delta_{X,S} = \deg \varphi \deg c_{n+1, H}^{H_s}(\Omega_{H/S})$$
 \end{T}
 Assuming resolution of singularities we find, using \cite{katosaito}:
 \begin{Cor} Suppose that $H$ is regular with special fiber $H_s$ and that $R$ has perfect residue field. Then $$\op{ord}_v \Delta_{X,S} = (-1)^n \op{Art}_{H/S}.$$
 \end{Cor}
 For polynomial equations and related discriminants we can give a result valid over general discrete valuation rings, possibly of mixed characteristics: 
 \begin{T} \label{thm:polynomial-discriminant} Suppose that $F$ is a polynomial in $n+1$ variables of degree $d$ with coefficients in a discrete valuation ring $R$. For the classical discriminant $\Delta_F$, we have
 $$\op{ord}_v \Delta_{F} = \deg \varphi \deg c_{d+1, H}^{H_s}(\Omega_{H/S})$$
 \end{T}
  We also provide the same type of formula for the Deligne discriminant introduced in a letter \cite{Deligne-Quillen} and studied in \cite{T.Saito-conductor}, and in particular remark that this discriminant is not always a discriminant but sometimes the power of the discriminant. We then turn to the case of a general discrete valuation ring and in the appendix give an overview of minimal models for degree $d$ hypersurfaces in $\bb P^n$ and show how to give the same type of multiplicity-formulas in terms of localized Chern classes. The main result is a generalization of Ogg's formula for general ternary forms of degree $d$. The final result states: %After having related this discriminant to Deligne's discriminant, we can compare this discriminant to the Deligne discriminant of a desingularization of the surface $X$ the ternary form defines. This discriminant can be calculated to be the Artin conductor, we show that it is possible to calculate the difference between the two discriminants, using T. Saito's result in \cite{T.Saito-conductor}.
\begin{T} Suppose that $F$ is given by a ternary form of degree $d$ with coefficients in a discrete valuation ring $R$ with spectrum $S = \Spec R$, satisfying the conditions of Bloch's formula, with associated scheme $X \to S$. Denote by $\Delta_F$ the corresponding discriminant. Suppose that the scheme $X$ is normal, then
$$\op{ord}_v \Delta_F = \sum_{x \in X_{sing}} \mu_{X,x} + \chi(X_{\overline{s}}) - \chi(X_{\overline{\eta}}) + \op{Sw}_{S'} H^1(X_{\overline{\eta}}, \bb Q_\ell).$$
where $\mu_{X,x} = 12 p_{g,x} + \Gamma_x^2 - b_1(E_x) + b_2(E_x)$ is an invariant of the normal surface singularity, and which in the complex geometric situation is equal to the Milnor number of a smoothing of $X$. The result also holds true when replacing the scheme $X$ defined by $F$ by a general proper flat normal scheme $X \to S$ with geometrically connected fibers, which is a local complete intersection, and the discriminant with the Deligne discriminant. \end{T}
This also contains as a special case Ogg's formula, since a minimal Weierstrass equation only has rational double points which will simplify the above terms considerably. In the complex geometric setting, when the special fiber has one component with isolated singularities, this resembles the formula in Proposition 1.2, Chapitre II of \cite{Teissier-cycles}. In the pure characteristic $p$ situation, with $X$ regular with reduced special fiber, a version of this result can also be found in \cite{Zink} for discriminants of versal deformations. \\ In the appendix I review some lemmas and properties of discriminants, and suggest that the semi-stable models of Koll\'ar \cite{Kollar-polynomials} might serve as natural models for generalizing Ogg's formula to higher dimensions. \\
$\mathbf{Acknowledgements}$: I wish to thank Gerard Freixas  for many interesting comments on various topics of this article and Marc-Hubert Nicole for his careful reading of the manuscript and for pointing out the reference \cite{Zink}. I'm also very grateful for the many explanations Jan Stevens shared with me on singularities, and in particular for pointing out a formula for the Milnor number of a normal surface singularity. This article also owes several improvements to remarks and suggestions by Qing Liu.
\section{Discriminants and localized Chern classes}
 For convenience of the reader we recall the following construction of \cite{blochcycles}. Suppose that $Y$ is an integral scheme of finite type over a regular scheme. Let $\mc E = [E \to E']$ be a two-term complex of vector bundles whose map is injective and whose cokernel is locally free of rank $n$ outside of a closed subscheme $Z$. Let $\CH_i(Y)$ denote the $k$-th Chow group of $X$ of algebraic cycles of dimension $i$ modulo rational equivalence. Then there is a localized Chern class for any $m > n$, which among other things induces a homomorphism $c_{m, Y}^Z(\mc E) \cap : \CH_i(Y) \to \CH_{i-m}(Z)$. This homomorphism only depends on the quasi-isomorphism class of $\mc E$. Now, let $Y$ be a puredimensional integral scheme $Y \to S$ be a flat projective \lci morphism of constant relative dimension $n$ which is generically smooth. Picking any factorization of this morphism $Y \hookrightarrow \bb P^M_S \to S$ the cotangent bundle $\Omega_{Y/S}$ has a 2-term resolution as above and outside of the singular locus $Z, i: Z \subseteq Y$ it is locally free of rank $n$. We then in particular have an element $c_{n+1, Y}^Z(\Omega_{Y/S}) := c_{n+1, Y}^Z(\Omega_{Y/S}) \cap [Y] \in \CH_{\dim Y - n - 1}(Z)$. It has the property that $i_* c_{n+1, Y}^Z(\Omega_{Y/S})  = c_{n+1}(\mc E) \in \CH_{\dim Y - n - 1}(Y)$. This moreover coincides with the "localized top Chern class" in \cite{intersection}.
We consider next a smooth geometrically integral projective variety $X \subseteq {\bb P}^M$ of dimension $n+1$ such that $\Delta_X$ is a hypersurface in $\check {\bb P}^M$.
\begin{T} Let $\varphi: \bb P(N) \to \Delta_X$ be the Gauss morphism. Then $$\varphi_* c_{n+1, \mc H}^{\bb P(N)}(\Omega_{\mc H/\check {\bb P}^M}) = \deg \varphi \cdot [\Delta_X]$$ in $\CH_{N-1}(\Delta_X) = \bb Z \cdot [\Delta_X].$
\end{T}
We have the following corollary which is true for dimension reasons.
\begin{Cor} \label{thm:Chern-discriminant} [compare \cite{Kleiman-enumerative}, III.8] Let $X \subseteq \bb P^M$ be such that $\Delta_X$ is a hypersurface. Consider the tautological hyperplane section $\mc H$ over $\check {\bb P}^M$. Then $$c_{n+1, \mc H}^{\bb P(N)}(\Omega_{\mc H/\check {\bb P}^M}) = [\bb P(N)] \in \CH_{M-1}(\bb P(N)) = \bb Z \cdot [\bb P(N)].$$
\end{Cor}

\begin{proof} (of the theorem) The hypersurface $\Delta_X$ is integral so it is obvious that $$\varphi_* c_{n+1, \mc H}^{\bb P(N)}(\Omega_{\mc H/\check{\bb P}^M}) = c [\Delta_X]$$ for some integer $c$. To determine $c$, it suffices to calculate the class of $\varphi_* c_{n+1, \mc H}^{\bb P(N)}(\Omega_{\mc H/\check{\bb P}^M})$ in $\CH_{M-1}(\check {\bb P}^M) = \op{Pic}(\check {\bb P}^M) = \bb Z$, where the map is the natural sending $\Delta_X$ to $\deg \Delta_X$. Denote by $L$ and $L'$ the natural tautological line bundles $\calo(1)$ on $X$ and $\check {\bb P}^M$. Then $\mc H$ is cut out by the section of $L \otimes L'$ on $X \times \check {\bb P}^M$ determined by the dual of $(L \otimes L')^{-1} \to \mc E \otimes \mc E^\vee \to \calo$ where $\mc E = \calo^{M}$. To compute $c$, we simply compute the class $p_* i_* c_{n+1, \mc H}^{\bb P(N)}(\Omega_{\mc H/\check {\bb P}^M}) = p_* c_{n+1}(\Omega_{\mc H/\check {\bb P}^M})$ in $ \op{Pic}(\check {\bb P}^M)$. We use the resolution
$$0 \to {L\otimes L'}|_\mc H^{-1} \to \Omega_{X}|_\mc H \to \Omega_{\mc H/\check {\bb P}^M} \to 0$$
and obtain by the Whitney sum formula:
$$c_{n+1}(\Omega_{\mc H/\check {\bb P}^M})) = \sum c_{n+1-i}(\Omega_X) c_1(LL')^{i} \cap [\mc H].$$
Since $c_1({L \otimes L'}) \cap [X \times \check {\bb P}^M] = [\mc H]$, we have
$$ \sum c_{n+1-i}(\Omega_X) c_1({L \otimes L'})^{i} = \sum c_{n+1-i}(\Omega_X) c_1({L \otimes L'})^{i+1} \cap [X \times \check {\bb P}^M].$$
Binomial expanding $c_1({L \otimes L'})^{i+1} = (c_1(L) + c_1(L'))^{i+1}$ we see that the only terms that will give a contribution after applying pushforward are the terms of the form $c_{n+1-i}(\Omega_X)(i+1)c_1(L)^ic_1(L')$. After rewriting we obtain that the class is given by a generic hyperplane times the number
$$(-1)^{n+1}\int_X \frac{c(T_X)}{(1+c_1(L))^2}$$
where $c(E) = 1 + c_1(E) + c_2(E) + \ldots $ is the total Chern class. This was calculated by Katz in \cite{SGA7-2}, and is equal to (correcting the typo in idem) $\deg \varphi \deg \Delta_X$ (also see \cite{GZK}, Chapter 2, Theorem 3.4, for a more suggestive formulation). We conclude the theorem.
\end{proof}
Notice that for a surface $X$ the Gauss morphism might not be birational. Indeed, by \cite{SGA7-2}, Expos\'e XVII, Proposition 3.5 and Corollaire 3.5.0 the Gauss morphism associated to a smooth surface is birational if and only if it is generically unramified if and only if there is a non-degenerate quadratic singularity in some hyperplane section. An explicit example of a smooth hypersurface whose dual is a hypersurface but the Gauss morphism completely ramified is given by a special Fermat hypersurface (cf. idem. (3.4.2)).\\
% and Zak's theorem in \cite{Fulton-Lazarsfeld}).
Whenever $\Delta_X$ is not a hypersurface, we have the following theorem:
\begin{Prop} Suppose that $m$ is the codimension of $\Delta_X$ in $\check{\bb P}^M$. Then
$$\varphi_* c_{n+m, \mc H}^{\bb P(N)}(\Omega_{\mc H/\check{\bb P}^M}) = c. [\Delta_X]$$
for some constant $c$.
\end{Prop}
It can be verified that $c=1$ in characteristic 0 (cf. \cite{Holme}), and by the same method should be computable as the degree of the Gauss morphism for the duals of some iterated generic hyperplane sections of $X$. \\
\begin{remark}In particular, if $\bb P^1$ is a general line through $P \in \Delta_X$, the above multiplicity is then the definition of the multiplicity of the point $P$ in $\Delta_X$ which was studied over the complex numbers in \cite{Dimca-Milnor}, \cite{Parusinki}, \cite{Nemethi} where formulas in terms of (generalized) Milnor numbers was given, and general formulas in terms of Segre classes was given in \cite{Aluffi}. In characteristic $p$, if we assume that the singularities are isolated, a general line through $P$ will define a smooth total space over a pencil, and the formula of Deligne \cite{SGA7-2}, Expos\'e XVI, Proposition 2.1 can be used together with general geometry of pencils to prove that the multiplicity of the discriminant is the total Milnor number in the sense of idem. The more general result when the total space is regular around the singular fibers is covered by Bloch's conjecture.\end{remark}
\begin{proof} (of Theorem \ref{thm:discriminant-different} and \ref{thm:polynomial-discriminant}) We will suppose we are also working over a field or the integers localized at $p$, $\bb Z_p$, and suppose Proposition \ref{prop:polynomial-discriminant-appendix}. Now, let be $R$ a discrete valuation ring and suppose we are given a morphism $f: \spec R \to \check{\bb P}^M$ such that the generic point doesn't intersect the discriminant locus and the image of the special point is $P$. Consider the closure $j: Z \to \check{\bb P}^M$ of the image of $\Spec R$ in $\check{\bb P}^M$ and suppose it has codimension $p$. Denote by $R'$ the local ring of $Z$ at $P$. By \cite{intersection}, Chapter 20 and \cite{IntviaAdams} Theorem C, we have a commutative intersection product $$\CH^{p}_Z(\check{\bb P}^M) \times \CH^{1}_{\Delta_X}(\check{\bb P}^M) \to \CH^{p+1}_{Z \cap \Delta_X}(\check{\bb P}^M)_{\bb Q}$$ which around $P$ has coefficient being the length $\ell_{R'} (R'/(\Delta))$, where $\Delta$ is some polynomial representative of $\Delta_X$ restricted to $Z$. We also have bivariant Chow groups with the same product
$$A^p_Z(\check{\bb P}^M) \times A^1_{\Delta_X}(\check {\bb P}^M) \to A^{p+1}_{Z \cap \Delta_X}(\check{\bb P}^M).$$
By Poincar\'e duality we have isomorphisms, possibly after tensoring with $\bb Q$, $A^i_Y(X) \simeq \CH_{\dim X - i}(Y) \simeq \CH_Y^i(X).$ By the arguments of \cite{IntviaAdams} 8.4 and \cite{intersection}, Corollary 17.4, these products respect the isomorphisms (here it is important that we are working over a discrete valuation ring to be able to apply loc. cit. 20.2). This in particular proves the commutativity of the product $$Z \cdot \varphi_* c_{n+1, \mc H}^{\bb P(N)}(\Omega_{\mc H/\check{\bb P}^M}) \cap [\mc H] = \varphi_{p^{-1} Z, *} c_{n+1, \mc H}^{\bb P(N)}(\Omega_{\mc H/\check{\bb P}^M}) \cap [p^*Z].$$
General properties of bivariant Chern classes now proves the formula around the local domain of dimension one $R'$. Finally, denote by $\tilde{R}' = \oplus {R'}_w$ the direct sum of the localizations of the normalization of $R'$. Since $R'$ is necessarily excellent, this is finite. One easily finds that $\op{ord}_w(\Delta) = \deg c_{n+1, \mc H_{{R'}_w}}^{\mc H_{k(w)}}(\Omega_{\mc H_{{R'}_w}/R'})$. For some $w$, the induced morphism $\Spec R \to \Spec {R'}_w$ is surjective and thus faithfully flat, and we can pull back the localized Chern class to obtain the searched for formula for the multiplicity of the discriminant in terms of the localized Chern class.

%By the established formula, we need to prove that the bivariant class evaluated on $Z$ is the same as the bivariant class evaluated on the total space intersected with $Z$. This amounts to the commutativity of bivariant Chow groups-intersections.

%general yoga of bivariant Chern classes shows that the Chern class pulls back along $Z$ $p_* c_{d+1, \mc H}^{\bb P(N)} (\Omega_{\mc H/\check{\bb P}^M}) \cap Z$ and the theory of multiplicities of \cite{IntviaAdams} provides that the multiplicity is controlled by the normalization of the local ring and is still equal to the localized Chern class. 

\end{proof}
\section{Families of curves and Deligne's discriminant}

Let $X$ be a smooth projective geometrically integral surface over a field $k$ and let $L$ be a very ample line bundle on $X$ defining an embedding $X \subseteq \bb P^M$. Consider the set $\bb P(N) = \{ (x, H), x \in (X \cap H)_{\op{sing}}$. The projection $f: \mc H \to \check{\bb P}^M$ is now relative curve whose singularities form the space $P(N)$ for $N$ the normal bundle of $X$ in $\bb P^M$ and the image of $P(N)$ in $\check {\bb P}^M$ is the discriminant of the embedding determined by $L$. If $k$ is of characteristic $0$, a theorem of Ein (cf. \cite{Ein}, partially attributed to Landman and Zak) the dual variety of a smooth surface is a hypersurface. The running hypothesis here is that $\Delta_X$ is a hypersurface.\\
 Consider, for a line bundle $A$ on $\mc H$, the isomorphism $$\det Rf_* A^{12} \simeq \langle \omega, \omega \rangle \langle A , A \omega^{-1} \rangle^6$$ over the locus away from the discriminant. Here $\det Rf_*$ denotes the determinant of the cohomology and $\langle A, B \rangle$ denotes the Deligne brackets which can be defined as the line bundle $\det Rf_* ((A - 1)(B-1))$ (cf. \cite{T.Saito-conductor} for a discussion on discriminants and the formalism introduced in \cite{determinant} for the Deligne brackets). The purpose of this section is to provide a natural interpretation of the discriminant as the degeneration of this rational isomorphism over $\check {\bb P}^M$. In \cite{Deligne-Quillen} Deligne calls this the discriminant section, and verifies that it corresponds to the usual discriminant in the case of degree $d$-curves in $\bb P^2$. \\
To calculate the degeneration we can suppose that $k$ is moreover algebraically closed. Let $\deg \Delta_X$ be the degree of $\Delta_X$ in $\check {\bb P}^M$. Then $$\deg \varphi \deg \Delta_X= \deg c_2(X) + 4g_H - 4 + \deg X$$ where $\varphi$ is the Gauss morphism as above. This follows immediately from the "class formula" $$\deg \varphi \deg \Delta_X = \chi(X) - 2 \chi(X \cap H_0) + \chi(X \cap H_0 \cap H_1)$$ in \cite{SGA7-2}, Expos\'e XVII, Proposition 5.7.2, for generic smooth hyperplane sections $H_0$ and $H_1$, and we have  $\chi(X) = \deg c_2(X), \chi(X \cap H_0) = 2 - 2g_H, \chi(X \cap H_0 \cap H_1) = \deg X$. \\
The following appears in \cite{Deligne-Quillen} under the headline "remarque inutile" in the special case of curves of degree $d$ in $\bb P^2$. The general, possibly surprising, result is that in fact the Deligne discriminant is not a discriminant but the power of a discriminant in the general setting. However, if one accepts that the localized Chern class is the discriminant times the degree of the Gauss morphism, this is not surprising since the Deligne-isomorphism specializes to the Riemann-Roch theorem in the Picard group where there is the same $c_2$-term which was calculated in the previous section. Thus this perhaps makes the following proposition into another remarque inutile, but it was also the motivating example behind this note.
\begin{Prop} [Remarque inutile] The Deligne-isomorphism extends to an isomorphism
$$\det Rf_* (A)^{12} \simeq \langle \omega, \omega \rangle \langle A , A \omega^{-1} \rangle^6 \otimes \calo(\deg \varphi \Delta_X)$$
over $\check{\bb P}^M$.
%$f_* c_2(\Omega_{W/\check {P}^n}) = \calo(d^*)$ in $\op{Pic} \check P^n = \bb Z \calo(1)$.
\end{Prop}
%where $\chi$ now denotes algebraic Euler characteristic:
\begin{proof} It is not difficult to prove that the order of degeneration is independent of the choice of line bundle. To calculate the degeneration we can assume that we really consider the isomorphism $\det Rf_* \omega^{12} \simeq \langle \omega, \omega \rangle$ over the smooth locus (the Mumford isomorphism). It is also clear that the degeneration is of the form $\calo(c \Delta_X)$ for the same reason as in the previous section. We need to calculate $c$, which we will do by calculating the degree of the various line bundles appearing in the isomorphism. Write $\mc L = L \otimes L'$ with notation as in the previous section. In this case we have by adjunction, $K = \omega := \omega_{\mc H/\check{\bb P}^M} = \omega_X \otimes {L \otimes L'}|_{\mc H}$ which admits a resolution
$$0 \to \omega_X \to \omega_X \otimes \mc L \to \omega_X \otimes \mc L|_\mc H \to 0$$ and so $$\det Rf_* \omega = \det Rp_* (\omega_X \otimes \mc L) \otimes \det Rp_* \omega_X^{-1} =   {L'}^{\rk(Rp_* \omega_X \otimes L)} \otimes \hbox{trivial sheaf}$$ and by Riemann-Roch for surfaces and the adjunction formula we have $$\rk(Rp_* ( \omega_X \otimes L)) = \chi(\omega_X \otimes L) = \frac{1}{2}(L+K)L + 1 + p_a = g_H-1 + 1+ p_a.$$
If we have a relative Cartier divisor $D$ (cf. \cite{Elkik}, III.2.6), then $\langle \calo(D), A, B \rangle \simeq \langle A|_D, B|_D \rangle$ for the multiple index Deligne brackets, so we also have $$\langle \omega, \omega \rangle = \langle \mc L, \omega_X \mc L, \omega_X \mc L \rangle_{X \times \check{\bb P}^M/\check{\bb P}^M}.$$ Using the formula $$\langle f^* A, B_1, \ldots, B_n \rangle = A^{B_1 . \ldots . B_n}$$ where the upper indices indicate the intersection number (cf. loc. cit. IV. 2. 1.a) we deduce that the latter line bundle is $L'$ to the power of $L^2 + 2KL + K^2+ 2L(L+K)$ (indicating the individual contributions from each term). This in turn means that the Mumford isomorphism induces an abstract isomorphism
$$\det Rf_* \omega^{12} \otimes \langle \omega, \omega \rangle^{-1} = \calo(\deg \varphi \Delta_X)$$
since the various powers are, by the Noether formula,
$$(6L(K+L) + K^2 + \deg c_2(X))-(L^2 + 2KL + K^2 + 2L(K+L)) $$ $$= \deg c_2(X) + 4g_H - 4 + \deg X  =  \deg \varphi \deg \Delta_X,$$
by the above remarks.
\end{proof}
\section{Case of ternary forms}
In this section, let $X$ be plane model of a projective plane curve over a discrete valuation ring $R$ with spectrum $S= \Spec R$ and perfect residue field. In particular $X$ is the scheme associated to $$F(X,Y,Z) = \sum_{i+j+k=d} a_{ijk}X^iY^jZ^k =0, a_{ijk} \in R.$$ We are interested in the order of the discriminant of this equation, and for the purposes of this section we can suppose $R$ is henselian with algebraically closed residue field. \begin{remark} All the results in this section carry over mutatis mutandis to the case when $X$ is a normal scheme, proper, flat and of local complete intersection over $S$ with geometrically connected general fibers, if we work with Deligne's discriminant. \end{remark}

By \cite{Liu-book}, Corollary 8.3.51, we have a desingularization $\pi: X' \to X$, i.e. $X'$ is regular and $\pi$ is a proper birational and an isomorphism over the regular locus. Both $X$ and $X'$ are local complete intersections so their dualizing sheaves are line bundles, and we write $\omega_{X'/S} = \pi^* \omega_{X/S} + \sum b_i E_i = \pi^* \omega_{X/S} + \Gamma$ for the exceptional divisors $E_i$ so that $\Gamma$ is the discrepancy. Using \cite{SGA7-2}, Expos\'e X, we have an intersection product $\Gamma^2$ and we define $p_g$ to be geometric genus of the singularities defined as the dimension of the $k(s)$-module $R^1 \pi_* \calo_{X'}$. The results in the previous section can then be extended over the integers using the argument of Proposition \ref{prop:polynomial-discriminant-appendix}, and we have:
\begin{T} Suppose $X$ is normal. Then $$\deg c_{2, X}^{X_s}(\Omega_{X/S}) - \deg c_{2, X'}^{{X'}_s}(\Omega_{X'/S}) = 12 p_g + \Gamma^2.$$
\end{T}
\begin{proof} The formula in $\cite{T.Saito-conductor}$, Theorem 3 (also see \cite{DRR3}, Theorem 4.11, for the formulation used here), implies that the difference is measured by the difference of the two Mumford "isomorphisms" which we write in the form
$$\lambda(\calo_{X'})^{12} \simeq \langle \omega', \omega' \rangle$$
and
$$\lambda(\calo_{X})^{12} \simeq \langle \omega, \omega \rangle.$$
The difference $\lambda(\calo_{X'})^{12} - \lambda(\calo_X)^{12}$ is calculated by considering the lengths of the cohomology groups of the cone of $\calo_X \to R\pi_* \calo_{X'}$. Since $X$ is normal and integral, by Zariski's main theorem the fibers of $\pi$ are connected so that $\calo_X \simeq \pi_* \calo_{X'}$ so this is just the length of the cohomology groups $R^1 \pi_* \calo_{X'}$ times 12. The contribution to the difference is $12 p_g$.
We now prove that
$$\langle \pi^* \omega, \pi^* \omega \rangle \simeq \langle \omega, \omega \rangle,$$
i.e. that the obvious isomorphism on the generic point extends to a global isomorphism. Indeed, by \cite{determinant}, $\det Rf_* (u \otimes v) \simeq \langle \det u, \det v \rangle$ for virtual bundles $u$ and $v$ of rank 0, and moreover if $u = a \otimes b$ for two virtual bundles of rank 0, then $\det u \simeq \calo$. If $u = v = \omega - 1$, then it is not difficult to see that the difference between the two bundles above is measured by $\det Rf_* (R^1 \pi_* \calo_{X'} \otimes u \otimes v)$, but $R^1\pi_* \calo_{X'}$ is clearly also of rank 0 and we conclude. Now, suppose $D$ is a 1-dimensional Cartier divisor supported on the special fiber $X'_s$, and $L$ is any line bundle on $X'$. Then $\langle D, L \rangle$ has a generic section induced by the trivialization $D_\eta = \emptyset$. We claim the order of this trivialization is the usual intersection number $D.L$ of \cite{SGA7-2}, Expos\'e X. We can suppose that $D$ is integral. In this case it follows from Riemann-Roch for singular curves (cf. \cite{intersection}, Example 18.3.4). We are now ready to calculate the difference $\langle \omega_{X'/S}, \omega_{X'/S} \rangle - \langle \omega_{X/S}, \omega_{X/S} \rangle$. By the above this is $\langle \omega_{X'/S}, \omega_{X'/S} \rangle - \langle \pi^* \omega_{X/S}, \pi^* \omega_{X/S} \rangle = 2 \omega_{X'/S} \Gamma - \Gamma^2 = \Gamma^2$
since $\omega_{X'/S} \Gamma =\pi^* \omega_{X/S} \Gamma + \Gamma^2 = \Gamma^2$ because $\pi^* \omega_{X/S} \Gamma=0 $ by projection formula and due to the fact that $\pi_* \Gamma$ is not a divisor.

\end{proof}
Since now $X'$ is a regular surface over $S$, using Bloch's formula in \cite{blochcycles} we have, since the generic fiber doesn't change, \begin{eqnarray*} \deg c_{2,X'}^{{X'}_s}(\Omega_{X'/S}) & = & -\op{Art}_{X'/S} = \chi(X'_{\overline s}) - \chi(X_{\overline \eta}) + \op{Sw} H^1(X_{\overline \eta}, \bb Q_\ell)  \end{eqnarray*} which gives
\begin{Cor} Under the above hypothesis we have $$\op{ord}_{v} \Delta = \deg c_{2,X}^{{X}_s}(\Omega_{X/S}) = 12 p_g + \Gamma^2  + \chi(X'_{\overline s}) - \chi(X_{\overline \eta})+ \op{Sw} H^1(X_{\overline \eta}, \bb Q_\ell).$$
\end{Cor}
Finally, we recall a formula of Leufer for the Milnor number of a normal surface singularity (cf. \cite{Laufer}), in the complex setting. For $x \in X$ an isolated singular point, $X$ a normal complex hyperursurface in $\bb C^3$, the Milnor number of $x$ in $X$ is known to be equal to $$\mu_{X, x} = 12 p_{g,x} + {\Gamma_x}^2 - b_1(E_x) + b_2(E_x)$$  where $X' \to X$ is a desingularization of the point $x$, the subscripts denoting the contribution related to each singular point and $b_i(E_x)$ denote the $\ell$-adic Betti numbers of the exceptional set $E_x$ over $x \in X$. We take this as the definition in the general case. Standard rewriting gives $\chi({X'}_{\overline s}) = \chi(\tilde X_{\overline s}) + \sum \chi(E_{x_i}) - \sum B(x)$ where $B(x)$ is the number of (geometric) branches of $x$ in $X_s$ under the map ${X'}_s \to X_s$ and $\tilde X_s$ is the strict transform of $X_s$. The formula $\chi(\tilde X_{\overline s}) - \chi(X_{\overline{s}}) = \sum (B(x) - 1)$ then implies:
\begin{Cor} With the above notation, $$\op{ord}_{v} \Delta = \sum_{x \in X} \mu_{X, x} + \chi(X_{\overline{s}}) - \chi(X_{\overline \eta}) + \op{Sw} H^1(X_{\overline \eta}, \bb Q_\ell).$$
\end{Cor}
% On one hand, we can rewrite $\Gamma^2$ in terms of the intersections of the various components using that $\pi^* X_s = \widetilde{X}_s + a_i E_i$ and the formula of Corollaire 1.8 in Expos X in \cite{SGA7-2}. Also, this formula looks very much like a formula of Steenbrink for the Milnor number of a smoothing of normal surface singularities whose reference escapes me at the moment. In case $X$ is already regular with finitely many singularities in the special fiber this actually already is the Milnor number.
The Betti numbers can be computed by a Mayer-Vietoris argument.  We find that $b_1(E) = 2g + b, b_2(E) = r$. Here $g = \sum g(\tilde{C}_i)$ is the sum of the genus of the normalizations of the reduced irreducible components $C_i$, and $b$ is the number of holes in the dual graph of a normal crossings model, defined as the graph whose vertices are the components of $Y$ and we connect two vertices by an edge for each intersection. Finally $r$ is the number of components of $E$. Since blowing up a regular surface in a closed point only produces rational curves and doesn't change the topology of the dual graph and all models in question are related by such a sequence of blowing ups and downs (cf. \cite{lichtenbaum}, Theorem 1.15, p. 392), $b_1(E)$ does not depend on the regular models $X'$ and only on $x \in X$. \\
It should finally be noted that if $X$ has only rational double point singularities then $\deg c_{2, X}^{X_s}(\Omega_{X/S}) = \deg c_{2, X'}^{{X'}_s}(\Omega_{X'/S})$ which is one of the miracles behind the usual formula of Ogg since the minimal Weierstrass equation is the unique Weierstrass model of a planar degree three curve which only has rational double points as singularities. For a true generalization of this formula it would be interesting to understand better the geometry of the semi-stable models of Koll\'ar (see the appendix).
%We conclude with an example:
%\begin{example}Let $p$ be a prime number and consider the Fermat curve $$X^p + Y^p + Z^p = 0$$ over $\bb Z_p$. Then the Swan conductor is equal to.... Using the formula for the discriminant in \cite{GZK}, chapter 13, Proposition 1.7 and 1.8, we find that $\Delta(X^p + Y^p + Z^p) = \pm p^{3p^3 - 10p^2 + 12p - 6}$. The standard computations for Euler characteristics gives $\chi(X_{\overline{s}}) - \chi(X_{\overline{\eta}}) = p^2 - 3p +2$. It is easy to verify that the associated scheme is normal with singular point
%\end{example}

\appendix

\section{Discriminant of homogenous degree $d$ polynomials in $n+1$ variables}
Let $K$ be a local field with respect to a discrete valuation with integer ring $R$. Let $F$ be a homogenous degree $d$ form in $n+1$ variables, defining a hypersurface in $\bb P^n_R$ which is generically smooth. This appendix deals with some general results on the discriminant of such polynomials, and remark a possible connection to minimal models in \cite{Kollar-polynomials}. We however first extend the the constructions and results in the first section over the integers (the same discussion goes through without problem when we consider Segre embedding multiprojective spaces such that the dual is a hypersurface). Consider the discriminant $\Delta_{d,n, \bb Q}$ in $\check {\bb P}^M_\bb Q$, $M = {n + d \choose d} - 1$ parameterizing singular hypersurfaces of degree $d$ in $\bb P^n_\bb Q$ and denote by $\Delta_{d,n}$ the Zariski/flat closure in $\check {\bb P}^M_\bb Z$. We then have the following proposition:
\begin{Prop} \label{prop:polynomial-discriminant-appendix}$[\Delta_{d,n}] = \varphi_* c_{n, \mc H}^{\bb P(N)}(\Omega_{\mc H/\check {\bb P}^M_\bb Z})$
\end{Prop}
\begin{proof} The fiber over $\bb F_p$ of $\Delta_{d,n}$ is the usual discriminant over the field $\bb F_p$. Since the localized Chern class is equal to the discriminant over $\bb Q$ by Theorem \ref{thm:discriminant-different} we have to verify that there are no vertical components appearing in the localized Chern class. This is a local question which we can verify by calculating both sides along certain special discrete valuation rings, and we can suppose that we are working over the $p$-adic integers $\bb Z_p$. But this is clear, since $\Delta_{d,n,\bb F_p}$ is not Zariski-dense in $\check {\bb P}^M$ we can find a $\overline{\bb F}_p$-point $s$ in $\check {\bb P}^M_\bb {F_p} \setminus \Delta_{d,n,\bb F_p}$ and thus also an unramified extension $K/\bb Q_p$ with integer ring $\calo_K$ with a $K$-point specializing to $s$. This corresponds to a smooth hypersurface of degree $d$ over $\calo_K$. We can evaluate the order of the discriminant and the degree of the localized Chern class on this $\calo_K$-point, and they are both zero. This proves the assertion.
\end{proof}
Thus the discriminant $\Delta_{d,n}$ is an integral polynomial defined up to sign and is controlled by the localized Chern class, and defines a flat scheme over $\bb Z$. The natural equivalence between models is projective equivalence: 
\begin{TL} Let $K$ be a field and suppose that $H = \{F =0 \}$ and $H' = \{F' = 0\}$ define $K$-isomorphic hypersurfaces in $\bb P^n$ which preserve a hyperplane section. Then there is a matrix $A \in GL_{n+1}(K)$ such that the induced $PGL_{n+1}(K)$-action on $\bb P^n$ carries $F$ to $F'$. %Moreover, writing $A$ in upper triangular form, the entries $a_{ij}$ are of the form $u^{k_i} s_{ij}$ where $k_i = \op{ord} x_i(D)$ \fixme{check what it should be} and $s_{ii} = 1$.
\end{TL}
\begin{proof} We can suppose that $d \geq 2$. We have an exact sequence
$$0 \to \calo(-d+1) \to \calo_{\bb P^n}(1) \to \calo_H(1) \to 0$$
which by the vanishing $H^0(\bb P^n, \calo(-d+1))= H^1(\bb P^n, \calo(-d+1)=0$ induces an isomorphism $$H^0(\bb P^n, \calo(1)) = H^0(H, \calo_H(1)).$$
By assumption of preserving a hyperplane we have an induced isomorphism $$H^0(H, \calo_H(1)) \simeq H^0(H', \calo_{H'}(1))$$ which then induces an automorphism of $H^0(\bb P^n, \calo(1))$. This is the sought after matrix.
\end{proof}
Given such a matrix, it acts on the cone $H^0(\check{\bb P}^n, \calo(d)) = Sym^d H^0(\check{\bb P}^n, \calo(1))$ over $\check{\bb P}^M$ but leaves the discriminant invariant. The discriminant is a hypersurface, given by $\Delta = 0$, and given a matrix $A$ acting on $\bb P^n$ we have $\Delta_A = m_A \Delta$ for some constant $m_A$. Moreover, $m_{A B} = m_A m_B$, so that $m = (\det A)^k$ for some $k$ and $k = d(d-1)^n$ (cf. \cite{Kollar-polynomials}, Lemma 4.5, using that the degree of this discriminant is $(n+1)(d-1)^n$, cf. \cite{GZK}, Example 4.15). For $d=3$ and $n=2$ this corresponds to the usual transformation rule for the discriminant of a Weierstrass equation. This discriminant can be computed as the resultant of the derivatives times a normalizing factor, $d^{((-1)^{n+1} - (d-1)^{n+1})/d}$ (cf. \cite{GZK}, Chapter 13, Proposition 1.7). The article \cite{Kollar-polynomials} approaches the question of finding "minimal models" of equations over discrete valuation rings, from the point of view of geometric invariant theory. These models are minimal, amongst other things, in the sense that the discriminant is minimal. It would be interesting to understand better the geometry of such minimal models.

\providecommand{\bysame}{\leavevmode\hbox to3em{\hrulefill}\thinspace}
\providecommand{\MR}{\relax\ifhmode\unskip\space\fi MR }
\providecommand{\MRhref}[2]{%
  \href{http://www.ams.org/mathscinet-getitem?mr=#1}{#2}
}
\providecommand{\href}[2]{#2}

\end{document}